\documentclass{amsart}

\usepackage{amsfonts,amssymb,amsmath,amsthm}
\usepackage{tikz}
\usepackage{tikz-cd}

\newcommand{\kk}{\mathbb{K}}

\newcommand{\m}{\mathbf{m}}

\newcommand{\D}{\Delta}
\newcommand{\N}{\mathbb{N}}
\newcommand{\cJ}{\mathcal{J}}
\newcommand{\cR}{\mathcal{R}}
\newcommand{\R}{\mathbb{R}}

\newcommand{\Z}{\mathbb{Z}}
\newcommand{\A}{\mathcal{A}}

\newtheorem{thm}{Theorem}[section]
\newtheorem{cor}[thm]{Corollary}
\newtheorem{lem}[thm]{Lemma}
\newtheorem{prop}[thm]{Proposition}
\newtheorem{ques}[thm]{Question}

\theoremstyle{definition}
\newtheorem{defn}[thm]{Definition}
\newtheorem{exm}[thm]{Example}

\theoremstyle{remark}
\newtheorem{remark}[thm]{Remark}

\usepackage{hyperref}

\begin{document}

\title{Generalized Splines and Graphic Arrangements}
\author{Michael DiPasquale}
\address{Michael DiPasquale\\     
	Department of Mathematics\\     
	Oklahoma State University\\     
	Stillwater\\
	OK \ 74078-1058\\     
	USA}     
\email{mdipasq@okstate.edu}
\urladdr{\url{http://math.okstate.edu/people/mdipasq/}}

\subjclass{Primary 13P20, Secondary 13D02, 32S22, 05E40}

\keywords{generalized splines, hyperhomology, graphic arrangements, multi-arrangements, logarithmic derivations, multi-derivations, chordal graphs}

\begin{abstract} 
We define a chain complex for generalized splines on graphs, analogous to that introduced by Billera and refined by Schenck-Stillman for splines on polyhedral complexes.  The hyperhomology of this chain complex yields bounds on the projective dimension of the ring of generalized splines.  We apply this construction to the module of derivations of a graphic multi-arrangement, yielding homological criteria for bounding its projective dimension and determining freeness.  As an application, we show that a graphic arrangement admits a free constant multiplicity iff it splits as a product of braid arrangements.
\end{abstract}

\maketitle

\section{Introduction}
A \textit{spline} is a piecewise polynomial function defined over a subdivision of a region by simplices or convex polytopes.  Study of spline spaces is a fundamental topic in approximation theory and numerical analysis (see~\cite{Boor}) while within the past decade geometric connections have been made between continuous splines and equivariant cohomology rings of toric varieties~\cite{Paynes}.  On the other hand, a \textit{hyperplane arrangement} is a collection of hyperplanes in $\R^n$.  One of the important invariants of a hyperplane arrangement is its module of logarithmic derivations; Terao~\cite{TeraoPoincare} shows that if this module is free then the Poincare polynomial of the arrangement complement completely factors (such arrangements are called \textit{free}).  Closely related is the module of multi-derivations, introduced in~\cite{ZieglerMulti}.  The module of multi-derivations is intimately linked to freeness of arrangements~\cite{YoshCharacterizationFreeArr}, yet even in the case of the braid arrangement this module is only partially understood~\cite{AbeSignedEliminable,TeraoMultiDer}.

In~\cite{HalSplit}, Schenck applies a result of Terao~\cite{TeraoMultiDer} to compute dimension formulas for classical splines on a subdivision of the $n$-simplex (the Alfeld Split). This is done by identifying the module of splines on this subdivision with a module of multi-derivations on the braid arrangement.  In this paper we extend Schenck's identification to subarrangements of the braid arrangement (called \textit{graphic arrangements}).  We use techniques from spline theory to characterize freeness and projective dimension of the module of multi-derivations on a graphic arrangement.

A natural language to use in making this identification is that of \textit{generalized splines}, recently introduced by Gilbert-Polster-Tymoczko~\cite{GSplines}.  Given an edge-labelled graph $(G,\m)$ where each edge $e$ is labeled by an ideal $\m(e)$ of an integral domain $R$, the ring of generalized splines $R_{G,\m}$ on this graph is a subring of the free $R$-module on the vertices consisting of tuples which satisfy congruence relations across each edge.  The main tool developed in this paper is a chain complex $\cR/\cJ[G]$ attached to any edge-labeled graph whose first cohomology module is the ring of generalized splines.  This chain complex is in the spirit of the chain complex introduced by Billera in~\cite{Homology} and later refined by Schenck-Stillman~\cite{LCoho} for classical splines.

The paper is organized as follows.  In \S~\ref{sec:GenSplines} we introduce generalized splines on graphs and define the chain complex $\cR/\cJ[G]$.  In \S~\ref{sec:Hyperhomology} we use hyperhomology to derive bounds on the projective dimension of the ring of generalized splines $R_{G,\m}$ (as an $R$-module) via the projective dimension of the homologies of $\cR/\cJ[G]$.  We introduce hyperplane arrangements in \S~\ref{sec:Preliminaries} and in \S~\ref{sec:DerivationComplex} we show that the module $D(\A_G,\m)$ of multi-derivations on a graphic arrangement $\A_G$ is naturally isomorphic to a ring of generalized splines.  We apply the results of \S~\ref{sec:Hyperhomology} to give the two following criteria:

\vspace{5 pt}

\noindent\textbf{Corollary~\ref{cor:free1}} $D(\A_G,\m)$ is free iff $H^i(\cR/\cJ[G])=0$ for all $i>0$.

\vspace{5 pt}

\noindent\textbf{Corollary~\ref{cor:pdimUB}}  Let $p_i=\mbox{pdim}(H^i(\cR/\cJ[G]))$.  Then 
\[
\mbox{pdim}(D(\A_G,\m))\le \max\limits_{i>0}\{p_i-i-1\},
\]
with equality if $H^i(\cR/\cJ[G])$ is nonzero for only a single $i>0$.

\vspace{5 pt}

In \S~\ref{sec:FreeMult} we use a characterization of chordal graphs due to Herzog-Hibi-Zheng to prove that $D(\A_G,\m)$ is free for a constant multiplicity $\ge 2$ iff the underlying essential arrangement splits as the product of braid arrangements ~\ref{thm:ProductsCompleteGraphs}.  This result bears resemblance to Abe-Terao-Yoshinaga's characterization of totally free arrangements as those which decompose as a product of one and two dimensional arrangements~\cite{TeraoTotallyFree}.  We close in \S~\ref{sec:Examples} with several computations illustrating how the complex $\cR/\cJ[G]$ may be used to study $\mbox{pdim}(D(\A_G))$.  The computations in this final section are inspired by~\cite[Problem~23]{SchenckComputationsConjectures}, where Schenck asks for a formula for $\mbox{pdim}(D(\A_G))$.

\section{Generalized Splines}\label{sec:GenSplines}

Let $R$ be an integral domain, $\mathcal{I}(R)$ the set of ideals of $R$, $G$ a finite simple graph (no loops or multiple edges) with vertices $V(G)$ and edges $E(G)$, and $\m:E(G)\rightarrow\mathcal{I}(R)$ an assignment of an ideal $\m(e)$ to each edge $e\in E(G)$.  We will assume the vertex set $V(G)=\{v_1,\ldots,v_k\}$ is ordered.  If an edge $e$ connects vertices $v_i$ and $v_j$ we denote $e$ by $\{i,j\}$ where $i<j$.  If $F\in\bigoplus_{v\in V(G)} R$, we denote by $F_i$ the component of $F$ corresponding to $v_i$.  In~\cite{GSplines}, Gilbert-Polster-Tymoczko give the following definition.

\begin{defn}\label{defn:GeneralizedSplines}
The ring of \textit{generalized splines} on the labelled graph $(G,\m)$ is defined by
\[
R_{G,\m}:=\{F\in \bigoplus_{v\in V(G)} R | F_i-F_j\in \m(e)\phantom{h} \forall e=\{i,j\}\in E(G)\}.
\]
\end{defn}

\noindent The simplicial coboundary map
\[
\delta^0: \bigoplus\limits_{v_i\in V(G)} Rv_i \rightarrow \bigoplus\limits_{e\in E(G)} Re
\]
from $0$-cochains of the graph $G$ to its $1$-cochains with coefficients in $R$ is defined on basis elements by
\[
\delta^0(v_i)=\sum\limits_{v\in e} c_{v_i,e} e
\]
where
\[
c_{v_i,e=\{j,k\}}=\left\lbrace
\begin{array}{rl}
+1 & \text{if } i=j\\
-1 & \text{if } i=k
\end{array}.
\right.
\]

\begin{lem}\label{lem:FirstCoho}
Let $(G,\m)$ be an edge-labeled graph with ordered vertex set.  The ring $R_{G,\m}$ is the kernel of the map
\[
\bigoplus\limits_{v\in V(G)} R \xrightarrow{\overline{\delta}^0} \bigoplus\limits_{e\in E(G)} R/\m(e),
\]
where $\delta^0$ is the coboundary map from vertices to edges and $\overline{\delta}^0$ is the quotient map.
\end{lem}

\begin{proof}
Let $F\in \bigoplus_{v\in V(G)} R$.  Then $F\in \mbox{ker}(\overline{\delta}^0)\iff F_i-F_j=0\in R/\m(e) \iff F_i-F_j\in \m(e)$ for every edge $e=\{i,j\}\in E(G)$.  The conclusion follows from Definition~\ref{defn:GeneralizedSplines}.
\end{proof}

In the case of classical splines, there is an underlying simplicial or polyhedral complex of which $G$ is the dual graph.  This allows for an extension of the map from Lemma~\ref{lem:FirstCoho} to a larger chain complex build off of the cellular chain complex of the underlying subdivision~\cite{Homology,LCoho}.  In the case of generalized splines the only a priori information which we have to carry out this process comes from the graph itself.

\begin{defn}\label{defn:clique}
Let $G$ be a graph with vertex set $V(G)$ and edge set $E(G)$.  A \textit{clique} of $G$ is a complete subgraph of $G$.  The \textit{clique complex} $\D(G)$ of $G$ is the simplicial complex on the vertex set $V(G)$ whose $i$-simplices are the cliques of $G$ with $(i+1)$ vertices.
\end{defn}

\begin{remark}
We denote by $\D(G)_i$ the set of cliques of $G$ with $(i+1)$ vertices, i.e. the simplices of $\D(G)$ of dimension $i$.
\end{remark}

\begin{defn}\label{defn:TopComplex1}
Let $G$ be a graph with $v$ vertices and $R$ a ring.  Define the chain complex $\cR[G]$ to be the co-chain complex of $\D(G)$ with coefficients in $R$, that is, $\cR[G]_i=\bigoplus\limits_{\gamma\in\D(G)_i} R$ and the differential $\delta^i: \cR[G]_i\rightarrow \cR[G]_{i+1}$ is the simplicial differential of the co-chain complex of $\Delta(G)$ with coefficients in $R$.
\end{defn}

\begin{remark}
Given an $i$-simplex $\sigma=\{j_0,\ldots,j_i\}, \delta^i(\sigma)=\sum_{\sigma\subset\tau} \pm \tau$,
where the sum runs over $\tau\in\D(G)_{i+1}$ such that $\sigma\subset\tau$.  The sign associated to the pair $(\sigma,\tau)$ is determined as follows.  Let $\tau=\{k_0,\ldots,k_{i+1}\}$.  Exactly one of the indices $k_0,\ldots,k_{i+1}$, say $k_t$, is not in $\sigma$.  Then the sign associated to $(\sigma,\tau)$ is $(-1)^t$.
\end{remark}

\begin{remark}
By definition $H^\bullet(\cR[G])\cong H^\bullet(\Delta(G);S)$, the singular cohomology of $\Delta(G)$ with coefficients in $R$.
\end{remark}

\begin{defn}\label{defn:Jideal}
	Let $(G,\m)$ be an edge-labeled graph.  Let $\sigma$ be an $i$-clique of $G$.  Then
	\[
	J(\sigma):=\sum\limits_{e\in E(\sigma)} \m(e).
	\]
	If $\sigma$ is a vertex of $G$, then $J(\sigma)=0$.
\end{defn}

\begin{defn}\label{defn:derivationComplex}
	Given an edge-labeled graph $(G,\m)$, $\cJ[G]$ is the sub-chain complex of $\cR[G]$ with $\cJ[G]_i=\bigoplus\limits_{\gamma\in\D(G)_i} J(\gamma)$.  $\cR/\cJ[G]$ denotes the quotient complex $\cR[G]/\cJ[G]$ with
	$\cR/\cJ[G]_i=\bigoplus\limits_{\gamma\in\D(G)_i} S/J(\gamma)$.
\end{defn}

\begin{remark}
There is a tautological short exact sequence of complexes $0\rightarrow \cJ[G] \rightarrow \cR[G] \rightarrow \cR/\cJ[G] \rightarrow 0$.
\end{remark}

\begin{lem}\label{lem:firstCohomology}
The ring of generalized splines of an edge-labeled graph $(G,\m)$ satisfies $R_{G,\m}\cong H^0(\cR/\cJ[G])$.	
\end{lem}
\begin{proof}
This is simply a restatement of Lemma~\ref{lem:FirstCoho}.
\end{proof}

\section{Hyperhomology and Projective Dimension of Generalize Splines}\label{sec:Hyperhomology}

In~\cite{Spect,Chow,CohVan}, Schenck and Schenck-Stiller use hyperhomology to obtain results on freeness and projective dimension of the module of algebraic splines.  The point is that for any complex $C_\bullet$, there are spectral sequences associated to the horizontal and vertical filtrations of the double complex formed by applying a right (or left) exact functor (in our case $\mbox{Hom}(\underline{\phantom{hh}},R)$) to a Cartan-Eilenberg resolution of $C_\bullet$, both of which converge to the left (or right) hyper-derived functors of $C_\bullet$.  See ~\cite[\S 5.7]{Weibel} for hyperhomology and~\cite{Spect} for the setup in the case of splines.

Suppose we are give a complex $C_\bullet$ with differential $d:C_i\rightarrow C_{i+1}$ increasing the indices.  Recall that a Cartan-Eilenberg resolution of $C_\bullet$ is defined by taking projective resolutions $P_{i\bullet}$ for each module $C_i$, with $P_{ij}$ being the $j^{th}$ projective module in a projective resolution of $C_i$, so that $B_{i\bullet}=H^i(P_{i\bullet})$ (cohomology taken with respect to the horizontal differential $d_h:P_{i,j}\rightarrow P_{i+1,j})$) is a projective resolution for $H^i(C_\bullet)$.  Let $\hat{P}_{\bullet\bullet}$ be the double complex with $\hat{P}_{ij}=\mbox{Hom}(P_{ij},R)$.

The terms on the $E_2$ page of the spectral sequence corresponding to the horizontal filtration (compute homology with respect to horizontal differential first) of $\hat{P}_{\bullet\bullet}$ are
\[
\phantom{i}_h E^2_{ij}=Ext^j_R(H^i(C_\bullet),R)
\]
while the terms on the $E_2$ page of the spectral sequence corresponding to the vertical filtration (compute homology with respect to vertical differential first) of $\hat{P}_{\bullet\bullet}$ are
\[
\phantom{i}_v E^2_{ij}=H^i(Ext^j_R(C_i,R)),
\]
where $H^i$ is computed with respect to the horizontal differential.  The differentials for the horizontal filtration on the $E^r$ page are of the form $\phantom{i}_h d^r_{ij}: \phantom{i}_h E^r_{ij}\rightarrow E^r_{i+r-1,j+r}$.  The differentials for the vertical filtration on the $E^r$ page are of the form $\phantom{i}_v d^r_{ij}: \phantom{i}_v E^r_{ij}\rightarrow E^r_{i-r,j-r+1}$.  With our grading conventions, the total complex $\hat{P}_{\bullet\bullet}$ has graded pieces $\bigoplus_{i-j=k} \hat{P}_{ij}$ labelled by the difference of $i$ and $j$.  The following result is essentially a translation of~\cite[Lemma~4.11]{Spect} into our context.

\begin{thm}\label{thm:pdimUB}
Let $(G,\m)$ be an edge-labelled graph, $\cR/\cJ[G]$ the associated chain comple, and suppose furthermore that $R/J(\sigma)$ is Cohen-Macaulay of codimension $i$ for each face $\sigma\in\D(G)_i$.  For $i>0$, set $p_i=\mbox{pdim}(H^i(\cR/\cJ[G]))$. Then $\mbox{pdim}(R_{G,\m})\le \max\limits_{i>0}\{p_i-i-1\}$, with equality if there is only a single nonvanishing $H^i(\cR/\cJ[G])$ for $i>0$.
\end{thm}
\begin{proof}
Note that if $R/J(\sigma)$ is Cohen-Macaulay of codimension $i$ for every $\sigma\in\D(G)_i$ then
\[
\cR/\cJ[G]_i=\bigoplus_{\sigma\in\D(G)_i} R/J(\sigma)
\]
is also Cohen-Macaulay of codimension $i$.

The terms $\phantom{i}_v E^1_{ij}=\mbox{Ext}^j_R(\cR/\cJ[G]_i,R)=\mbox{Ext}^j_R(\oplus_{\sigma\in\D(G)_i} R/J(\sigma))$ vanish except when $i=j$ since $\cR/\cJ[G]_i$ is Cohen-Macaulay of codimension $i$.  This implies that the vertical filtration collapses to yield $\mbox{Ext}^j_R(\bigoplus_{\sigma\in \D(G)_i} R/J(\sigma))=\phantom{i}_vE^1_{ij}=\phantom{i}_vE^{\infty}_{ij}$.  It follows that the homology of the total complex is concentrated in degree zero (corresponding to the diagonal $i=j$).

Comparing with the horizontal filtration, $\phantom{i}_h E^\infty_{ij}=0$ unless $i=j$; in particular $\phantom{i}_h E^\infty_{0j}=0$ for $j>0$ and the successive quotients $\phantom{i}_h E^r_{0j}$ of $\phantom{i}_h E^2_{0j}=\mbox{Ext}^j_R(R_{G,\m},R)$ eventually stabilize at zero.  Hence $\mbox{Ext}^j_R(H^0(\cR/\cJ[G]),R)$ can only be nonzero if one or more of the terms $\phantom{i}_h E^r_{r-1,j+r}\neq 0$.  These latter terms are subquotients of $\phantom{i}_hE^2_{r-1,j+r}=\mbox{Ext}^{j+r}_R(H^{r-1}(\cR/\cJ[G]))$.  So we see that $\mbox{Ext}^j_R(R_{G,\m},R)\neq 0$ implies that $\mbox{Ext}^{j+r}_R(H^{r-1}(\cR/\cJ[G]))\neq 0$ for some $r\ge 2$.  Since $p=\mbox{pdim}(R_{G,\m})$ is the largest index of a nonvanishing $\mbox{Ext}^p_R(R_{G,\m},R)$, one of the modules $\mbox{Ext}^{p+r}_R(H^{r-1}(\cR/\cJ[G]))$ is nonzero for some $r\ge 2$.  Setting $p_i=\mbox{pdim}(H^i(\cR/\cJ[G]))$, it follows that $p\le \max\{p_i-i-1\}$.

Now suppose that $H^i(\cR/\cJ[G])=0$ except for $i=k>0$.  Then $\phantom{i}_hE^r_{0,j}=\mbox{Ext}^j_R(R_{G,\m},R)$ and $\phantom{i}_hE^r_{kj}=\mbox{Ext}^j_R(H^k(\cR/\cJ[G]))$ for $r=2,\ldots,k+1$, since the differentials are trivial for the range $2\le r\le k-1$.  On page $k+1$ we have $\mbox{Ext}^j(R_{G,\m},R)=\phantom{i}_hE^{k+1}_{0,k+1}\cong\phantom{i}_hE^{k+1}_{k,j+k+1}\cong \mbox{Ext}^{j+k+1}(H^k(\cR/\cJ[G]))$, since the spectral sequence collapses on the $(k+1)^{st}$ page.  It follows that $\mbox{pdim}(R_{G,\m})=\mbox{pdim}(H^k(\cR/\cJ[G]))-k-1$.
\end{proof}

The strongest statement we can obtain (in terms of freeness) is the following characterization due to Schenck-Stiller~\cite[Theorem~3.4]{CohVan}.

\begin{thm}\label{thm:SchenckFree}
Let $S$ be a polynomial ring over a field $\kk$ of characteristic $0$.  Let $(G,\m)$ be a labelled graph with the following two properties.
\begin{itemize}
\item $S/J(\sigma)$ is Cohen-Macaulay of codimension $i$ for every $\sigma\in\D(G)_i$
\item $H^i(\cR/\cJ[G])$ is supported in codimension $\ge i+2$ for all $i>0$
\end{itemize}
Then $S_{G,\m}$ is free iff $H^i(\cR/\cJ[G])=0$ for all $i>0$.
\end{thm}

\begin{proof}
A proof can be found in~\cite{Spect}.
\end{proof}

\section{Hyperplane Arrangements}\label{sec:Preliminaries}

Let $\kk$ be a field of characteristic $0$, $V$ a $\kk$-vector space, and $V^*$ the dual vector space.  A hyperplane arrangement $\A\subset V$ is a union of hyperplanes $H=V(\alpha_H)\subset V$, where $\alpha_H\in V^*$.  The \textit{rank} of a hyperplane arrangement $\A\subset V$ is given by $\mbox{rk}(\A)=\dim V-\dim(\cap_{H\in \A} H)$.  The arrangement $\A\subset V$ is called \textit{essential} if $\mbox{rk}(\A)=\dim V$ and \textit{central} if $\cap_i H_i\neq\emptyset$.  If $\A$ is central, denote by $\A^e$ the essential arrangement obtained by quotienting out the vector subspace $W=\cap_{H\in\A} H$.

The intersection lattice $L_\A$ of $\A$ is the lattice whose elements (flats) are all possible intersections of the hyperplanes of $\A$, ordered with respect to reverse inclusion.  This is a ranked lattice, with rank function the codimension of the flat.  Given a flat $X$, the (closed) subarrangement $\A_X$ is the hyperplane arrangement of those hyperplanes of $\A$ which contain $X$.

If $\A\subset V$ is an arrangement, let $S=\mbox{Sym}(V^*)$ and $Q=\prod_{H\in\A} \alpha_H$, the defining equation of $\A$.  The module of derivations of $\A$, denoted $D(\A)$, is defined by
\[
D(\A)=\{\theta\in\mbox{Der}_\kk(S)| \theta(\alpha_H)\in\langle \alpha_H \rangle\mbox{ for all } H\in\A \}.
\]
If $D(\A)$ is free as an $S$-module, we say $\A$ is free.
\begin{defn}
A multi-arrangement $(\A,\m)$ is an arrangement $\A\subset V$, along with a function $\m:\A\rightarrow \Z_{>0}$ assigning a positive integer to every hyperplane.  Let $S=\mbox{Sym}(V^*)$.  The module of multi-derivations $D(\A,\m)$ is defined by
\[
D(\A,\m)=\{\theta\in\mbox{Der}_\kk(S)|\theta(\alpha_H)\in\langle \alpha_H^{\m(H)} \rangle\mbox{ for all } H\in\A\}
\]
\end{defn}
If $D(\A,\m)$ is free as an $S$-module then we say that the multi-arrangement $(\A,\m)$ is free and $\m$ is a \textit{free multiplicity} of $\A$.  If $m$ is a positive integer and $\m(H)=m$ for every $H\in \A$, then $(\A,\m)$ is denoted $\A^{(m)}$ and $D(\A,\m)$ is denoted $D^{(m)}(\A)$ to be consistent with~\cite{TeraoMultiDer}.

\begin{lem}\label{lem:DerivationMatrix}
Let $(\A,\m)$ be a multi-arrangement in $V\cong \kk^n$.  Let $\alpha_i$ be the form defining the hyperplane $H_i$, and set $m_i=\m(H_i)$.  The module $D(\A,\m)$ of multiderivations on $\A$ is isomorphic to the kernel of the map
\[
\psi:R^{n+d}\rightarrow R^d,
\]
where $\psi$ is the matrix
\[
\begin{pmatrix}
& \vline & \alpha_1^{m_1} & & \\
B & \vline  & & \ddots & \\
& \vline & & & \alpha_k^{m_k}
\end{pmatrix}
\]
and $B$ is the matrix with entry $B_{ij}=a_{ij}$, where $\alpha_j=\sum_{i,j} a_{ij} x_i$.
\end{lem}
\begin{proof}
See the comments preceding~\cite[Theorem~4.6]{DimSeries}.
\end{proof}

\begin{prop}\label{prop:pdimLB}
Let $(\A,\m)$ be a multi-arrangement, $X\in L_\A$, and $(\A_X,\m_X)$ the corresponding closed subarrangement with restricted multiplicities.  Then \\
$\mbox{pdim}(D(\A,\m))\ge \mbox{pdim}(D(\A_X,\m_X))$.
\end{prop}
\begin{proof}
This is~\cite[Proposition~1.7]{AbeSignedEliminable}.
\end{proof}

\begin{lem}[Ziegler~\cite{ZieglerMulti}]\label{lem:globalUB}
For any arrangement $\A\subset V,\mbox{pdim}(D(\A,\m))\le \mbox{rk}(\A)-2$.  In particular, if $\mbox{rk}(\A)\le 2$ then $(\A,\m)$ is free.
\end{lem}
\begin{proof}
Let $k=\mbox{rk}(\A)$, $W=\cap_{H\in\A} H$, and $S=\mbox{Sym}(V^*)$.  Let $l_1,\ldots,l_{\dim V-k}\in S$ be a maximal choice of linearly independent forms not vanishing on $W$.  These form a regular sequence on $D(\A,\m)$.  Modding out by these gives a polynomial ring in $k=\mbox{rk}(\A)$ variables.  In this new ring, the image $\overline{D(\A,\m)}$ is still the kernel of the matrix $\psi$ from Lemma~\ref{lem:DerivationMatrix}, hence is a second syzygy.  The result now follows from the Hilbert syzygy theorem.
\end{proof}

Given two multi-arrangements $(\A_1,\m_1),(\A_2,\m_2)$ with $\A_1\subset V_1,\A_2\subset V_2$, the product $(\A_1,\m_1)\times(\A_2,\m_2)$ is the multi-arrangement $(\A_1\times\A_2,\m)$ in $V_1\oplus V_2$ with underlying arrangement $\A_1\times\A_2=\{H_1\oplus V_2|H\in\A_1\}\cup\{V_1\oplus H_2| H_2\in\A_2\}$, and multiplicities $\m(H_1\oplus V_2)=\m_1(H_1)$, $\m(V_1\oplus H_2)=\m_2(H_2)$.  If $\A$ does not split as the product of nontrivial subarrangements, then $\A$ is \textit{irreducible}.

\begin{lem}\label{lem:DerSum}
Given multi-arrangements $(\A_1,\m_1)$ and $(\A_2,\m_2)$, $D((\A_1,\m_1)\times(\A_2,\m_2))\cong D(\A_1,\m_1)\oplus D(\A_2,\m_2)$.
\end{lem}
\begin{proof}
See~\cite[Lemma~1.4]{TeraoCharPoly}.
\end{proof}

\subsection{Graphic multi-arrangements}
Let $G$ be a graph with no loops or multiple edges, so that each edge of $G$ is uniquely defined by its incident vertices (we allow isolated vertices).  Let $E(G)$ denote the set of edges of $G$ and $V(G)$ the set of vertices of $G$.  Set $c(G)$ to be the number of connected components of $G$, and $v=|V(G)|$.  The \textit{graphic} arrangement corresponding to $G$ is
\[
\A_G=\bigcup_{\substack{{i,j}\in E(G)\\ i<j} } V(x_i-x_j)\subset\R^v.
\]

\begin{remark}
If $(\A_G,\m)$ is a graphic multi-arrangement, the function $\m: \A_G \rightarrow \N_{>0}$ can be viewed as a function $\m:E(G)\rightarrow \N_{>0}$. If we label the vertices of $G$ by integers and $e=\{i,j\}\in E(G)$ then we will write $\alpha_e$ or $\alpha_{ij}$ for the form $x_i-x_j$ and $m_e$ or $m_{ij}$ for the value of $\m$ on the hyperplane $V(\alpha_e)$.
\end{remark}

Note that $\mbox{rk}(\A_G)$ is the number of edges in a spanning tree of $G$ (spanning forest if $G$ is not connected).  Hence $\mbox{rk}(\A_G)=v-c(G)$.  Applying Lemma~\ref{lem:globalUB} we obtain

\begin{cor}\label{cor:globalUB}
Let $(\A_G,\m)$ be a graphic multi-arrangement on a graph $G$ with $c(G)$ components.  Then $\mbox{pdim}(\A_G)\le |V(G)|-c(G)-2$.  In particular, if every connected component of $G$ has at most three vertices, then $(\A_G,\m)$ is free.
\end{cor}

\begin{defn}\label{defn:inducedSubgraph}
A subgraph $H\subset G$ is an \textit{induced subgraph} if every edge in $E(G)$ connecting vertices of $H$ is in $E(H)$.  A \textit{partition} of $G$ is a disjoint union of induced subgraphs of $G$ containing all the vertices of $G$.
\end{defn}

\begin{remark}
Disjoint unions of induced subgraphs of $G$ are in one to one correspondence with equivalence relations on $V(G)$, or partitions of $V(G)$, justifying the terminology.
\end{remark}

\begin{lem}\label{lem:GraphArrLattice}
The set of all partitions of $G$ form a lattice under containment which is isomorphic to $L_{\A_G}$.	
\end{lem}
\begin{proof}
A flat of $L_{\A_G}$ is obtained from a partition of $G$ by intersecting hyperplanes corresponding to edges of the partition.  Likewise, given a flat $X$ of $L_{\A_G}$, define an equivalence relation $\sim_X$ on the vertices of $G$ by $s\sim_X t$ if $s$ is connected to $t$ by a sequence of edges $\{i,j\}$ so that $x_i-x_j$ vanishes on $X$.  The partition of the vertices corresponding to the equivalence classes of $\sim_X$ yields the corresponding union of induced subgraphs of $G$.  Furthermore, containment of partitions corresponds to reverse containment of the corresponding flats, so the lattice of partitions is isomorphic to $L_{\A_G}$.
\end{proof}

\begin{remark}\label{rem:indsubFlat}
Given a partition $H$ of $G$, the closed subarrangement $(\A_G)_H$ (viewing $H$ as a flat) is almost the same as the arrangement $\A_H$ (viewing $H$ as a graph).  The only difference is the ambient space: $(\A_G)_H$ is an arrangement in $\kk^{|V(G)|}$ while $\A_H$ is in $\kk^{|V(H)|}$.  For convenience we will make the convention that, whenever $H$ is a partition of $G$, $\A_H$ refers to $(\A_G)_H$.
\end{remark}

\begin{cor}\label{cor:pdimLBGraph}
Let $(\A_G,\m)$ be a graphic multi-arrangement and $H$ a partition of $G$.  Then $\mbox{pdim}(D(\A_G,\m))\ge \mbox{pdim} (D(\A_H,\m_H))$.
\end{cor}
\begin{proof}
$H$ corresponds to a flat in $L_{\A_G}$ by Lemma~\ref{lem:GraphArrLattice}, so by Proposition~\ref{prop:pdimLB} $\mbox{pdim}(\A_G)\ge \mbox{pdim} (\A_G)_H$.
\end{proof}

\begin{remark}\label{rem:Product}
If $G_1$ and $G_2$ are graphs with multiplicities $\m_1,\m_2$ and $G_1\sqcup G_2$ is their disjoint union with multiplicity function $\m$ restricting to $\m_1$ on $G_1$ and $\m_2$ on $G_2$, then $(\A_{G_1\sqcup G_2},\m)=(\A_{G_1},\m_1)\times(\A_{G_2},\m_2)$.
\end{remark}

Recall that the \textit{blocks} of $G$ are the maximal connected components with no cut vertex.

\begin{lem}\label{lem:Product}
Let $G$ be a connected graph.  Then
\begin{enumerate}
\item $\A^e_G$ is irreducible iff $G$ does not have a cut vertex.
\item Let $G_1,\ldots,G_k$ be the blocks of $G$.  Then $\A^e_G=\A^e_{G_1}\times\cdots\times \A^e_{G_k}$ is the decomposition of $\A^e_G$ into irreducible factors.
\end{enumerate}
\end{lem}
\begin{proof}
(1): The arrangement $\A^e_G$ splits as a product iff the underlying cycle matroid of $G$ is not connected.  The cycle matroid of $G$ is connected iff the graph itself is $2$-vertex connected, i.e. it has no cut vertex.
(2): The decomposition $\A^e_G=\A^e_{G_1}\times\cdots\times \A^e_{G_k}$ follows from the decomposition of the cycle matroid of $G$ as a sum of connected components, which correspond to the blocks of $G$.
\end{proof}

\section{Derivations and Generalized Splines}\label{sec:DerivationComplex}

Our primary results in this section are Corollaries~\ref{cor:free1} and~\ref{cor:pdimUB}, giving homological criteria for freeness of $(\A_G,\m)$ and bounds on the projective dimension of $D(\A_G,\m)$, respectively.  We accomplish this by identifying $D(\A_G,\m)$ with a ring of generalized splines.

\begin{prop}\label{prop:DerSplines}
Let $(\A_G,\m)$ be a graphic multi-arrangement.  Let $\m'$ be the edge-labeling assigning the principal ideal $\m'(e)=\langle \alpha_e^{\m(e)} \rangle\subset S$ to each edge $e\in E(G)$.  Then $D(\A_G,\m)\cong S_{G,\m'}$.
\end{prop}

\begin{proof}
In the case of the graphic arrangement $\A_G$, the rows of the matrix $\psi$ in Lemma~\ref{lem:DerivationMatrix} are labeled by edges of $G$.  Note that the submatrix $B$ is the transpose of the simplicial boundary map from edges to vertices of $G$.  Hence $\theta\in \bigoplus_{v\in V(G)} S$ is in $D(\A_G,\m)$ iff $\theta_i-\theta_j$ is a polynomial multiple of $\alpha_{ij}^{\m_{ij}}$ for every edge $e=\{i,j\}\in E(G)$.  By Definition~\ref{defn:GeneralizedSplines}, this is the same as the condition for $\theta$ to be a generalized spline on the edge-labeled graph $(G,\m')$.
\end{proof}

From now on we associate to the graphic mult-arrangement $(\A_G,\m)$ the chain complex $\cR/\cJ[G]$ associated to the edge-labeled graph $(G,\m')$, with $\m'$ as in Proposition~\ref{prop:DerSplines}.  By Proposition~\ref{prop:DerSplines} and Lemma~\ref{lem:firstCohomology}, $D(\A_G,\m)\cong H^0(\cR/\cJ[G])$.

In order to apply the results of \S~\ref{sec:Hyperhomology}, we analyze the codimension of the modules $H^i(\cR/\cJ[G])$.  The arguments closely follow those in~\cite{AssHom} for the Schenck-Stillman chain complex for classical splines on polyhedral complexes.

\begin{defn}
	Let $G$ be a graph on $v$ vertices and $H\subseteq G$ a subgraph.  Define $I(H)$ to be the ideal of $S=\kk[x_1,\ldots,x_v]$ generated by $\{x_i-x_j|\{i,j\}\in E(H)\}$.
\end{defn}

\begin{defn}
	Let $G$ be a graph and $I$ an ideal of $S=\kk[x_1,\ldots,x_v]$.  Define an equivalence relation $\sim_I$ on $G$ by $s\sim_I t$ iff $s$ is connected to $t$ by a path all of whose edges $\{i,j\}$ satisfy $x_i-x_j\in I$.  Let $G_I$ be the union of induced subgraphs of $G$ corresponding to equivalence classes of $\sim_I$.
\end{defn}

\begin{lem}\label{lem:localize}
	Let $(\A_G,\m)$ be a graphic multi-arrangement on a graph $G$ with $v$ vertices and set $S=\kk[x_1,\ldots,x_v]$.
	Let $P\subset S$ be a prime.  Then $H^i(\cR/\cJ[G])_P=H^i(\cR/\cJ[G_P])_P$ for $i>0$.
\end{lem}
\begin{proof}
	We use the fact that localization commutes with taking cohomology so it suffices to show that the localized complexes $\cR/\cJ[G]_P$ and $\cR/\cJ[G_P]_P$ agree in cohomological degree $>0$.  If $\sigma\in\D(G)_i$, where $i>0$, then $J(\sigma)\subset P$ iff $x_i-x_j\in P$ for every edge $e\in\sigma$ iff $\sigma\in\D(G_P)_i$.  It follows that, for $i>0$,
	\[
	(\cR/\cJ[G]_i)_P=\bigoplus\limits_{\sigma\in\D(G)_i} \left(\dfrac{S}{J(\sigma)}\right)_P=\bigoplus\limits_{\sigma\in\D(G_P)_i} \left(\dfrac{S}{J(\sigma)}\right)_P=(\cR/\cJ[G_P]_i)_P,
	\]
	and the result follows.
\end{proof}

\begin{thm}\label{thm:assPrimes}
Let $(\A_G,\m)$ be a graphic multi-arrangement.  If $i>0$ then \\
$\mbox{codim}(H^i(\cR/\cJ[G]))\ge i+2.$
\end{thm}
\begin{proof}
Let $P\subset S$ be a prime in the support of $H^i(\cR/\cJ[G])$.  It suffices to show that $\mbox{rk}(\A_{G_P})\ge i+2$, since then $I(G_P)\subset P$ has codimension at least $i+2$.  Suppose that $\dim \D(G_P)\le i-1$.  Then $H^i(\cR[G_P])$ vanishes and $P$ cannot be in the support of $H^i(\cR/\cJ[G])$ by Lemma~\ref{lem:localize}.  In particular, since $\dim \D(G_P)\le \mbox{rk}(\A_{G_P})$, $\mbox{rk}(\A_{G_P})\ge i-1$.

Now suppose $\dim \D(G_P)=i$ and $\mbox{rk}(\A_{G_P})=i$ or $i+1$.  We claim that $\D(G_P)$ is contractible or a union of contractible simplicial complexes.  First, if $\mbox{rk}(\A_{G_P})=i$ then $G_P$ is a complete graph on $i+1$ vertices and $\D(G_P)$ is an $i$-simplex, clearly contractible.  If $\mbox{rk}(\A_{G_P})=i+1$, then $G_P$ has a complete subgraph $K$ on $(i+1)$ vertices.  There must be one other edge in a spanning forest for $G_P$.  Disregarding isolated vertices, if $G_P$ is not connected then it consists of $K$ along with an edge connecting two vertices.  If $G_P$ is connected, then it consists of $K$ along with one other vertex $v$ joined to some number of vertices of $K$.  In this latter case $\D(G_P)$ consists of two simplices (corresponding to $K$ and the neighbors of the vertex $v$) joined along a common face.  This is clearly contractible.

By the preceding paragraph, $H^i(\cR[G_P])=0$ for $i>0$ if $\dim \D(G_P)=i$ and $\mbox{rk}(\A_{G_P})=i$ or $i+1$.  The tail end of the long exact sequence associated to $0\rightarrow \cJ[G_P] \rightarrow \cR[G_P] \rightarrow \cR/\cJ[G_P] \rightarrow 0$ yields that $H^i(\cR/\cJ[G_P])=0$ as well, hence $P$ is not in the support of $H^i(\cR/\cJ[G])$ by Lemma~\ref{lem:localize}.

Now suppose $\dim\D(G_P)=i+1$ and $\mbox{rk}(\A_{G_P})=i+1$.  Then $G_P$ is a complete graph on $i+2$ vertices and $\D(G_P)$ is an $(i+1)$-simplex.  In this case $H^i(\cR[G_P])=H^{i+1}(\cR[G_P])=0$ and the long exact sequence corresponding to $0\rightarrow \cJ[G_P] \rightarrow \cR[G_P] \rightarrow \cR/\cJ[G_P] \rightarrow 0$ yields
\[
H^i(\cR/\cJ[G_P])\cong H^{i+1}(\cJ[G_P]).
\]
However, $H^{i+1}(\cJ[G_P])$ is the cokernel of
\[
\bigoplus\limits_{\tau\in \D(G_P)_i} J(\tau)\rightarrow J(\sigma),
\]
where $J(\sigma)=\langle l_e^{m_e}|e\in G_P \rangle$.  This map is surjective, hence $H^i(\cR/\cJ[G_P])= H^{i+1}(\cJ[G_P])=0$.  So $\mbox{rk}(\A_{G_P})\ge i+2$.
\end{proof}

\begin{cor}\label{cor:free1}
The graphic multi-arrangement $(\A_G,\m)$ is free iff $H^i(\cR/\cJ[G])$ $=0$ for all $i>0$.
\end{cor}
\begin{proof}
We must show that the two bulleted hypotheses from Theorem~\ref{thm:SchenckFree} hold for the complex $\cR/\cJ[G]$.  First, let $\sigma\in\D(G)_i$ have vertices $v_0,\ldots,v_i$.  Then
\[
J(\sigma)=\langle (x_i-x_j)^{m_{ij}} |\{i,j\}\in\sigma \rangle
\]
is $I(\sigma)=\langle x_1-x_0,\ldots,x_i-x_0 \rangle$-primary.  Since $I(\sigma)$ is a complete intersection of codimension $i$, $S/J(\sigma)$ is Cohen-Macaulay of codimension $i$.  By Theorem~\ref{thm:assPrimes}, $H^i(\cR/\cJ[G])$ is supported in codimension $\ge i+2$ for all $i>0$, so we are done.
\end{proof}

\begin{cor}\label{cor:pdimUB}
Let $(\A_G,\m)$ be a graphic multi-arrangement.  For $i\ge 1$, set $p_i=\mbox{pdim}(H^i(\cR/\cJ[G]))$.  Then $\mbox{pdim}(D(\A_G,\m))\le \max\limits_{i\ge 1} \{p_i-i-1\}$, with equality if there is only a single nonvanishing $H^i(\cR/\cJ[G])$ for $i>0$.
\end{cor}

\begin{proof}
Since we have already shown that $J(\sigma)$ is Cohen-Macaulay, this follows directly from Theorem~\ref{thm:pdimUB}.
\end{proof}

\section{Graphic Arrangements Admitting a Free Constant Multiplicity}\label{sec:FreeMult}

Let $\A_G$ be a graphic arrangement.  The main result of this section is that $\A_G$ admits a free constant multiplicity iff its blocks are the maximal complete subgraphs of $G$.  We start by analyzing cycles.

\begin{thm}\label{thm:Cycles}
Suppose $G$ is a cycle with $v>3$ vertices.  Then $\mbox{pdim}(D(\A_G,\m))$ $=v-3$ for any choice of $\m$.
\end{thm}
\begin{proof}
In this case $\D(G)=G$, hence $H^1(\cJ[G])=\bigoplus_{\{i,j\}\in E(G)} J(ij)$, $H^1(\cR[G])$ $=H^1(G;S)=S,$ and $H^i=0$ for $i>1$ for all three chain complexes $\cJ[G],\cR[G],$ and $\cR/\cJ[G]$.  Via the long exact sequence arising from $0\rightarrow \cJ[G] \rightarrow \cR[G] \rightarrow \cR/\cJ[G] \rightarrow 0$ we have
\[
H^1(\cR/\cJ[G])=\mbox{coker}(H^1(\cJ[G])\rightarrow S)=\dfrac{S}{\sum_{\{ij\}\in E(G)} J(ij)}
\]
It follows that $H^1(\cR/\cJ[G])$ is $I(G)=\langle \alpha_e | e\in E(G)\rangle$-primary.  As such, it is Cohen-Macaulay with $\mbox{pdim}(H^1(\cR/\cJ[G]))=\mbox{codim}(H^1(\cR/\cJ[G]))=\mbox{codim}(I(G))=v-1$.  Since this is the only nonvanishing cohomology, it follows from Corollary~\ref{cor:pdimUB} that $\mbox{pdim}(D(\A_G,\m))=v-3$, the maximum possible by Corollary~\ref{cor:globalUB}.
\end{proof}

As a corollary to Theorem~\ref{thm:Cycles} and Corollary~\ref{cor:pdimLBGraph} we obtain a generalization of ~\cite[Corollary~2.4]{HalKung} to multi-arrangements.

\begin{cor}\label{cor:pdimLBCycle}
Let $(\A_G,\m)$ be a graphic multi-arrangement and let $m$ be the length of the longest induced cycle of $G$.  Then $\mbox{pdim}(D(\A_G,\m))\ge m-3$.
\end{cor}

Recall that a graph is \textit{chordal} if every cycle of length $>3$ has a chord, that is, an edge joining two vertices of the cycle.  Equivalently, every induced cycle has length three.

\begin{cor}\label{cor:Chordal}
Let $(\A_G,\m)$ be a graphic multi-arrangement.  If $D(\A_G,\m)$ is free, then $G$ is chordal.
\end{cor}

\begin{remark}
It follows from Stanley~\cite{StanleySupersolvable} that a graphic arrangement $\A_G$ is free iff $G$ is chordal.  Hence Corollary~\ref{cor:Chordal} could be stated as $(\A_G,\m)$ is free $\implies$ $\A_G$ is free.  This is not true for general hyperplane arrangements; it was conjectured to be true by Ziegler~\cite[Conjecture~13]{ZieglerMulti} but a counterexample was produced in~\cite{ReinerCounterEx}.
\end{remark}

\begin{remark}\label{rem:Yosh}
Yoshinaga has shown~\cite[Proposition~4.1]{YoshExtendability} that generic arrangements are totally non-free.  Since cycles of length $>3$ yield generic graphic arrangements, Corollary~\ref{cor:Chordal} also follows from this more general result.
\end{remark}

To characterize graphs admitting a free constant multiplicity we use a description of chordal graphs due to Herzog-Hibi-Zheng.

\begin{defn}\label{defn:quasiForest}
Let $\D$ be a simplicial complex.  A facet $F$ of $\D$ is a \textit{leaf} if there is a facet $H$ of $\D$ so that, for any facet $G\neq F$ of $\D$, $G\cap F\subset H\cap F$.  A \textit{leaf ordering} of $\D$ is an ordering $\{F_1,\ldots,F_k\}$ of the facets of $\D$ so that, for $i=1,\ldots,k$,  $F_i$ is a \textit{leaf} of the subcomplex of $\D$ generated by $\{F_1,\ldots,F_i,F_i\}$.  If $\D$ has a leaf ordering, then it is called a \textit{quasi-forest}.
\end{defn}

\begin{thm}[Herzog-Hibi-Zheng~\cite{Herzog}]\label{thm:chordalEquiv}
A graph $G$ is chordal iff $\D(G)$ is a quasi-forest.
\end{thm}

\begin{thm}\label{thm:ProductsCompleteGraphs}
Let $G$ be a graph and $m\ge 2$.  The following are equivalent.
\begin{enumerate}
\item $\A^{(m)}_G$ is free.
\item $\D(G)$ has a leaf ordering $\{F_1,\ldots,F_k\}$ so that $F_{i+1}\cap \{F_1,\ldots,F_i\}$ consists of at most a single vertex for $i=1,\ldots,k-1$.
\item The blocks of $G$ are the maximal complete subgraphs of $G$.
\item $\A^e_G$ is the product of its subarrangements 
\[
\{\A^e_H|H\subset G\mbox{ a maximal complete subgraph}\}.
\]
\end{enumerate}
\end{thm}
\begin{proof}
We prove the equivalence of these conditions under the assumption that $G$ is connected.  The full result then follows from Remark~\ref{rem:Product}.

$(1)\implies (2)$:
Suppose $D^{(m)}(\A_G)$ is free.  By Corollary~\ref{cor:Chordal}, $G$ is chordal and hence $\D(G)$ has a leaf ordering $\{F_1,\ldots,F_k\}$ by Corollary~\ref{thm:chordalEquiv}.  Suppose $F_{i+1}\cap \Delta_i=\{F_1,\ldots,F_i\}$ contains an edge $\{s,t\}$ of $F_{i+1}$.  Since $F_{i+1}$ is a leaf of $\Delta_{i+1}$, there is a facet $H\subset \Delta_i$ so that $H\cap F_{i+1}=\Delta_i\cap F_{i+1}$.  Pick a vertex $v\in H\setminus F_{i+1}$ and $w\in F_{i+1}\setminus H$.  Then the subgraph $K$ of $G$ induced by the four vertices $v,w,s,t$ is isomorphic to the complete graph on four vertices with an edge removed.  This is the so-called \textit{deleted} $A_3$ \textit{arrangement} and a characterization of its free multiplicities due to Abe~\cite{AbeDeletedA3} shows that it does not admit any free constant multiplicity $>1$.  It follows from Corollary~\ref{cor:pdimLBGraph} that $G$ can have no such induced subgraph.  Hence $F_{i+1}\cap\{F_1,\ldots,F_i\}$ consists of at most a single vertex.

$(2)\implies (3)$:
The assumption in (2) implies that the blocks of $G$ (maximal subgraphs without a cut vertex) correspond to the facets of $\D(G)$, which are exactly the maximal complete subgraphs of $G$.

$(3)\implies (4)$:  This follows from Lemma~\ref{lem:Product}.

$(4)\implies (1)$: A seminal result of Terao~\cite{TeraoMultiDer} states that $\A^{(m)}$ is free if $\A$ is a Coxeter arrangement.  The graphic arrangement of a complete graph on $n$ vertices is the braid arrangement $A_{n-1}$, which is Coxeter.  It follows that each of the factors in the product of (3) are free for any constant multiplicity.  Now the result follows from Lemma~\ref{lem:DerSum}.
\end{proof}

\begin{ques}\label{ques:4}
Is there an extension of Theorem~\ref{thm:ProductsCompleteGraphs} to subarrangements of other Coxeter arrangements?
\end{ques}

\begin{cor}\label{cor:totallyfree}
	Let $G$ be a graph.  Then the following are equivalent.
	\begin{enumerate}
		\item $(\A_G,\m)$ is free for all $\m$
		\item $\dim\D(G)\le 2$ and $\D(G)$ has a leaf ordering $\{F_1,\ldots,F_k\}$ so that $F_{i+1}\cap \{F_1,\ldots,F_i\}$ consists of at most a single vertex for $i=1,\ldots,k-1$.
		\item The blocks of $G$ are complete subgraphs on at most three vertices.
		\item $\A^e_G$ is the product of its subarrangements 
		\[
		\{\A^e_H|H\subset G\mbox{ a complete subgraph with at most three vertices}\}.
		\]
	\end{enumerate}
\end{cor}
\begin{proof}
	For $(2)\implies(3) \implies (4)\implies (1)$ we proceed as in the proof of Theorem~\ref{thm:ProductsCompleteGraphs}.  For $(1)\implies (2)$, we already know that $(\A_G,\m)$ free for all $\m$ implies $\D(G)$ has the required leaf ordering, just not the dimension restriction.  For this it suffices to produce a nonfree multiplicity on the $A_3$ braid arrangement (complete graph on four vertices).  An example of such a multiplicity may be found in~\cite[Example~5.13]{EulerMult}.
\end{proof}

\begin{remark}
Corollary~\ref{cor:totallyfree} is a special case of a result of Abe-Terao-Yoshinaga ~\cite{TeraoTotallyFree} that a multi-arrangement $(\A,\m)$ is free for all $\m$ iff $\A$ is the product of one and two dimensional arrangements.
\end{remark}

\section{Examples}\label{sec:Examples}

In~\cite[Problem~23]{SchenckComputationsConjectures}, Schenck asks for a formula for the projective dimension of $D(\A_G)$ for graphic arrangements.  In~\cite{HalKung}, Kung and Schenck ask for a characterization of graphs for which the bound $\mbox{pdim}(D(\A_G))\le m-3$, where $m$ is the length of a longest induced cycle, is tight.  We call this the \textit{Kung-Schenck bound}.  In this section we explore these two problems three examples.

\begin{exm}[Wheel]\label{exm:Wheel}
	Let $W_n$ be the graph with $n$ vertices on an exterior cycle all joined to a central vertex.  Figure~\ref{fig:Wheel} shows $W_4$.
	\begin{figure}[htp]
		\centering
		
		\begin{tikzpicture}[scale=1]
		\tikzstyle{dot}=[circle,fill=black,inner sep=1 pt];
		
		\node[dot] (1) at (0,0) {};
		\node[dot] (2) at (1,0){};
		\node[dot] (3) at (0,1){};
		\node[dot] (4) at (-1,0){};
		\node[dot] (5) at (0,-1){};
		
		\draw (1)node[above right]{$0$}--(2)
		(1)--(3)
		(1)--(4)
		(1)--(5);
		\draw (2)node[right]{$1$}--(3) node[above] {$2$}--(4)node[left]{$3$} -- (5) node[below]{$4$}--(2);
		\end{tikzpicture}
		\caption{$W_4$}\label{fig:Wheel}
	\end{figure}
	
	For the simple arrangement $\A_{W_n}$, we show that the Kung-Schenck bound is tight.  Since $\D(W_n)$ is contractible, $H^i(\cR[W_n])=0$ for $i=1,2$, $H^2(\cR/\cJ[W_n])$ $=0$, and $H^1(\cR/\cJ[W_n])\cong H^2(\cJ[W_n])$.  The complex $\cJ[W_n]$ has the form
	\[
	0\rightarrow \bigoplus_{e\in E(W_n)} J(e) \rightarrow \bigoplus_{\sigma\in\D(W_n)_2} J(\sigma)
	\]
	Arguing as in~\cite[Lemma~3.8]{LCoho}, there is a presentation for $H^2(\cJ[G])$ given by
	\[
	\sum_{\sigma\in \D(W_n)_2} \mbox{syz}(J(\sigma)) \xrightarrow{i} \bigoplus_{e\in E(W_n)} S(-1),
	\]
	where $\mbox{syz}(J(\sigma))$ denotes the module of syzygies on $J(\sigma)$ and $\mbox{coker}(i)=H^2(\cJ[W_n])$.  In this case $\mbox{syz}(J(\sigma))$ is generated by a relation in degree zero and a choice of Koszul syzygy in degree one.  All generators corresponding to boundary edges of $W_n$ are killed in the cokernel and generators corresponding to interior edges are equivalent.  Hence $H^1(\cJ[G])$ is principal, generated in degree one, and is killed by any form $\alpha_e$ where $e$ is a boundary edge.  More precisely,
	\[
	H^1(\cR/\cJ[G])\cong H^2(\cJ[G])\cong S(-1)/\langle \alpha_e| e\in \partial\D(W_n) \rangle S(-1)
	\]
	Hence $H^1(\cJ[G])$ has projective dimension $n-1$ and $D(W_n)$ has projective dimension $n-3$ by Corollary~\ref{cor:pdimUB}.
\end{exm}

\begin{exm}[Fritsch Graph]\label{exm:Fritsch}
	We look for graphs that fail the Kung-Schenck bound for topological reasons.  A natural place to look is graphs $G$ on $v$ vertices for which $\D(G)$ is a simplicial sphere.  In this case the long exact sequence from $0\rightarrow \cJ[G] \rightarrow \cR[G] \rightarrow \cR/\cJ[G] \rightarrow 0$ yields $H^2(\cR/\cJ[G])\cong S/I(G)$, so $\mbox{pdim}(H^2(\cR/\cJ))=v-1$.  In such a case it could happen that the differential on the $E_3$ page of the Cartan-Eilenberg spectral sequence yields an isomorphism between $Ext^{v-1}_S(H^2(\cR/\cJ[G]),S)$ and $Ext^{v-4}_S(D(\A_G),S)$, or at least a nontrivial map between subquotients of these.
	\begin{figure}[htp]
		\centering
		\begin{tikzpicture}[scale=1]
		\tikzstyle{dot}=[circle,fill=black,inner sep=1 pt];
		
		\node[dot] (1) at (0,0) {};
		\node[dot] (2) at (1,0){};
		\node[dot] (3) at (1,-1){};
		\node[dot] (4) at (-1,-1){};
		\node[dot] (5) at (-1,0){};
		\node[dot] (6) at (0,-2){};
		\node[dot] (7) at (0,3){};
		\node[dot] (8) at (-3,-3){};
		\node[dot] (9) at (3,-3){};
		
		\draw (1)--(2)--(3)--(1);
		\draw (1)--(4)--(3);
		\draw (1)--(5)--(4);
		\draw (1)--(7)--(5);
		\draw (2)--(7)--(9)--(2);
		\draw (3)--(6)--(9)--(3);
		\draw (6)--(8)--(9);
		\draw (6)--(4)--(8);
		\draw (5)--(8);
		\draw (7)--(8);
		\end{tikzpicture}
		\caption{$G$}\label{fig:Fritsch}
	\end{figure}
	Consider the graph $G$ in Figure~\ref{fig:Fritsch}.  This graph has $9$ vertices, so $\mbox{pdim}(H^2(\cR/\cJ[G]))=8$.  If a nontrivial differential on the $E_3$ page exists, then we expect $D(\A_G)$ to have projective dimension $5$.  Instead, computations in Macaulay2 show that $\mbox{pdim}(\A_G)=2$, which is precisely the Kung-Schenck bound since the longest length of an induced cycle is five.  According to Macaulay2, there is an isomorphism between $Ext^6_S(H^1(\cR/\cJ[G]),S)$ and $Ext^8_S(H^2(\cR/\cJ[G]),S)$; this isomorphism on the $E_2$ page kills the possible contribution to $\mbox{pdim}(D(\A_G))$ on the $E_3$ page.  This also gives an example where the upper bound from Corollary~\ref{cor:pdimUB}, $\mbox{pdim}(D(\A_G))\le 5$, is far from tight.
\end{exm}

\begin{ques}\label{ques:1}
	Building on Example~\ref{exm:Wheel} and Example~\ref{exm:Fritsch}, is the Kung-Schenck bound tight for all graphs whose clique complex is a spherical or planar triangulation?
\end{ques}

\begin{ques}\label{ques:2}
	Is $\mbox{pdim}(D(\A_G,\m))$ determined entirely by $H^1(\cR/\cJ[G])$?
\end{ques}

\noindent In terms of the Cartan-Eilenberg spectral sequence, Question~\ref{ques:2} could be phrased by asking whether all nonvanishing modules $Ext^i_S(D(\A_G,\m))$ are accounted for by nontrivial differentials on the $E_2$ page.

\begin{exm}\label{exm:TriangularPrism}
	\begin{figure}[htp]
		\centering
		
		\begin{tikzpicture}[scale=1]
		\tikzstyle{dot}=[circle,fill=black,inner sep=1 pt];
		
		\node[dot] (1) at (0,0) {};
		\node[dot] (2) at (0,-1){};
		\node[dot] (3) at (1,-2){};
		\node[dot] (4) at (-1,-2){};
		\node[dot] (5) at (-1,1){};
		\node[dot] (6) at (1,1){};
		
		\draw (1)node[above]{$1$}--(5) node[above]{$5$}--(6) node[above]{$6$};
		\draw (2)node[below]{$2$}--(3) node[below] {$3$}--(4)node[below]{$4$};
		\draw (2)--(3)--(6)--(1);
		\draw (1)--(2)--(4)--(5);
		\end{tikzpicture}
		\caption{$G$}\label{fig:TriPrism}
	\end{figure}
	
	Consider the graph $G$ in Figure~\ref{fig:TriPrism}, which appears in~\cite[Example~2.6]{HalKung}.  This is one of the simplest graphs where the Kung-Schenck bound is not tight.
	
	The complex $\cR[G]$ has the form
	\[
	S^6\rightarrow S^9 \rightarrow S^2.
	\]
	$\cJ[G]$ has the form
	\[
	\bigoplus_{ij\in E(G)}J(ij) \rightarrow J(045)\oplus J(123).
	\]
	$\cR/\cJ[G]$ has the form
	\[
	S^6 \rightarrow \bigoplus_{ij\in E(G)} S/J(ij) \rightarrow S/J(123)\oplus R/J(045)
	\]
	
	The complex $\D(G)$ is homotopy equivalent to a wedge of two circles, hence
	\[
	H^i(\cR[G])=\left\lbrace
	\begin{array}{rl}
	S & i=0\\
	S^2 & i=1\\
	0 & i\ge 2
	\end{array}
	\right.
	\]
	Since the two cliques of $G$ on three vertices are disjoint, it is straightforward to see that $H^i(\cJ[G])=0$ for $i\ge 2$.  Hence we have
	\[
	H^1(\cR/\cJ[G])=\mbox{coker}(H^1(\cJ[G]) \rightarrow H^1(\cR[G])=S^2).
	\]
	Macaulay2~\cite{M2} shows that $H_1(\cR/\cJ[G])$ has projective dimension $4$.  By Corollary~\ref{cor:pdimUB}, $\mbox{pdim}(D(\A_G))=2$.  $H_1(\cR/\cJ[G])$ has three associated primes of the form $I(H)$ where $H$ is an induced cycle of length four.  These have codimension three.  In this case, we may associate the failure of exactness of the Kung-Schenck bound on $\mbox{pdim}(D(\A_G))$ to the fact that $H_1(\cR/\cJ[G])$ is not Cohen-Macaulay of codimension three.
\end{exm}

\section{Acknowledgements}

I thank Stefan Tohaneanu, Takuro Abe, and Max Wakefield for providing feedback on earlier drafts of the paper.  I especially thank Stefan Tohaneanu for the suggestion of extending the results of~\cite{HalSplit} to graphic arrangements and Takuro Abe for providing Remark~\ref{rem:Yosh}.  I am grateful to Jay Schweig for patiently explaining some of the basics of matroid theory.  I am always thankful for the guidance and advice of Hal Schenck.

\bibliography{GraphBib,SplinesBib}{}
\bibliographystyle{plain}

\end{document}